\documentclass{article}
\usepackage[utf8]{inputenc}
\usepackage{fullpage}
\usepackage{hyperref,verbatim}
\usepackage{enumitem,bm}
\usepackage{url}
\usepackage{wrapfig}
\usepackage{amsthm,amsmath,amssymb}
\usepackage{xcolor}
\usepackage{pgf,tikz,tkz-graph,subcaption}
\usepackage{tkz-berge}
\usepackage{caption}
\usepackage{graphicx}
\usepackage{subfloat}
\usetikzlibrary{arrows,shapes}
\usetikzlibrary{decorations.pathreplacing}
\usepackage{algorithm}
\usepackage{bm,color}
 \usepackage{array}
 \usepackage[capitalise]{cleveref}
\usepackage{cite}
\usetikzlibrary{circuits.logic.US,circuits.logic.IEC}
\allowdisplaybreaks
\usetikzlibrary{decorations.markings,decorations.pathreplacing}
\tikzstyle{vertex}=[circle, draw, inner sep=0pt, minimum size=4pt]

\definecolor{DarkGreen}{rgb}{0.2, 0.6, 0.3}

\newtheorem{defi}{Definition}
\newtheorem{conj}[defi]{Conjecture}
\newtheorem{cor}[defi]{Corollary}
\newtheorem{thr}[defi]{Theorem}
\newtheorem{lem}[defi]{Lemma}

\newtheorem{prop}[defi]{Proposition}

\newtheorem{question}[defi]{Question}
\newtheorem{remark}[defi]{Remark}
\newtheorem{claim}[defi]{Claim}
\newcommand*{\myproofname}{Proof}
\newenvironment{claimproof}[1][\myproofname]{\begin{proof}[#1]}{\end{proof}}

\newcommand*{\floorceil}[1]{\left\lfloor #1\right\rceil}
\newcommand*{\bceil}[1]{\left\lceil #1\right\rceil}
\newcommand*{\bfloor}[1]{\left\lfloor #1\right\rfloor}
\newcommand*{\abs}[1]{\lvert #1\rvert}

\def \d{\, \mathrm{d}}

\newcommand{\ecc}{ecc}
\newcommand{\diam}{diam}

\newcommand{\rad}{rad}
\def \b{\beta}

\title{On the main distance-based entropies: the eccentricity- and Wiener-entropy}
\author{Stijn Cambie \thanks{Department of Computer Science, KU Leuven Campus Kulak-Kortrijk, 8500 Kortrijk, Belgium. Supported by the Institute for Basic Science (IBS-R029-C4) and a postdoctoral fellowship by the Research Foundation Flanders (FWO) with grant number 1225224N. Email: \protect\href{mailto:stijn.cambie@hotmail.com}{\protect\nolinkurl{stijn.cambie@hotmail.com}}}
\and
Yanni Dong\thanks{Corresponding author. School of Mathematics and Statistics,
Northwestern Polytechnical University, P.R. China, and Faculty of Electrical Engineering, Mathematics and Computer Science, University of Twente, The Netherlands, supported by National Natural Science Foundation of China (12071370,~12131013 and~U1803263) and China Scholarship Council (202006290070),
E-mail: {\tt y.dong@utwente.nl}.
}
}

\date{}

\begin{document}

\maketitle
\begin{abstract}
    We define the Wiener-entropy, which is together with the eccentricity-entropy one of the most natural distance-based graph entropies.
    By deriving the (asymptotic) extremal behaviour, we conclude that the Wiener-entropy of graphs of a given order is more spread than is the case for the eccentricity-entropy.
    We solve $3$ conjectures on the eccentricity-entropy and give a conjecture on the Wiener-entropy related to some surprising behaviour on the graph minimizing it.
\end{abstract}

\section{Introduction}\label{sec:int}

Subsection~\ref{sub:background} gives the general background about Shannon entropy (and is less related with the main story). 
In Subsection~\ref{subsec:distbasedentropies}, we explain our notation and give an overview of the studied elementary distance-based graph entropies.
The maximum entropy is usually $\log_2(n)$ and as such an elementary overview for the maximum graph entropy is given in Subsection~\ref{subsec:maxI_G}. Finally, the overview of our contributions is summarized in Subsection~\ref{subsec:contrib}.

\subsection{Background}\label{sub:background}
In this paper, we pay attention to two graph entropies which are based on Shannon entropy. As a foundational concept in information theory, Shannon entropy~\cite{Shannon} is a concept to measure uncertainty of a random variable.
This is a general mathematical notion of entropy that also can be used to define e.g. Boltzmann entropy in statistical physics.
Let $X$ be a discrete random variable with set of possible outcomes $\{x_1,x_2,\ldots,x_n\}$, and let  $P=\{p(x_1),p(x_2),\ldots,p(x_n)\}$ be the probability distribution on $X$. The entropy of $X$ is defined by
\begin{align*}
-\sum_{i=1}^{n}p(x_i)\log_{2}(p(x_i)).
\end{align*}
We omit the subscript $2$ and assume $\log$ refers to the base-2 logarithm throughout. We use $\ln$ to denote the natural logarithm, with base Euler's number $e$.

Because Shannon entropy has some conventional properties, such as the subadditivity and monotonicity, it can be used as a useful tool to solve some classical graph theory problems. K\"orner \cite{KornerMR0354162} discovered that the subadditivity of graph entropy in connection with a problem in coding theory can be used to prove the bounds for graph covering problems.  A bound for the perfect hashing was proposed by K\"orner and Marton \cite{Korner} based on an extension of graph entropy to hypergraphs. Csisz\'ar et al. \cite{CsiszarMR1075064} gave a new characterization of a perfect graph by applying the subadditivity property of graph entropy. Newman and Wigderson \cite{NewmanMR1361388} proved lower bounds on the formula size of Boolean functions by using entropy on hypergraphs. For more information on graph entropy and its applications, we refer the reader to a survey by Simonyi~\cite{SimonyiMR1338619}.

One of three great challenges for half-century-old computer science, proposed by Brooks Jr. \cite{Brooks10.1145/602382.602397}, is to provide a metric for
the information embodied in structure. The initial idea to measure a graph by extending the Shannon entropy is to consider the orbits of the graph as a variable \cite{RashevskyMR72417}. The vertices are topologically equivalent if they belong to the same orbit of the graph. Since then, to measure graphs, a lot of entropy-based notions appeared by replacing the probabilities in the Shannon entropy by graph invariants. Braustein  et al. \cite{BraunsteinMR2284272} proposed the notion of von Neumann entropy of a graph based on Shannon entropy constructed from the eigenvalues of the density matrix of the graph. To unify some existing graph measures based on Shannon entropy, Dehmer \cite{Dehmerfunctionals} defined a general formula for graph entropy based on information functionals. Related to the degree sequence of a graph, Cao, Shi and Dehmer \cite{CaoMR3212009} proposed degree-based entropy.
The first degree-based entropy is the principal degree-based entropy of a graph. Recently, some extremal problems for this entropy have been proven in~\cite{CDM22,CM22+,CM22+2}.
We refer the interested reader to the two survey papers \cite{DehmerMR2737460,LW16} and two books \cite{ComplexiMR2229139,Mathematical} for more complexity measures regarding graph entropies.


\subsection{Distance-based graph entropies}\label{subsec:distbasedentropies}

Before introducing the main objects, we need to introduce some related terminology and notation. The term graph represents a simple, connected, finite and undirected graph throughout the paper. Let $G=(V,E)$ be a graph. The {\em distance} between two vertices $u$ and $v$, denoted by $d(u,v)$, is the length of a shortest path from $u$ to $v$. In particular, it satisfies the {\em triangle inequality} $d(u,v)\leq d(u,w)+d(w,v)$ for all vertices $u,v,w\in V$.
The {\em eccentricity} of a vertex $v$, denoted by $ecc(v)$, is the maximum distance from $v$ to any other vertex  (i.e., $ecc(v)=\max\{d(v,u): u\in V\}$). A {\em central vertex} of $G$ is a vertex with the minimum eccentricity. 
The {\em diameter} $\diam(G)$ of $G$ is the maximum eccentricity among all vertices of $G$. The {\em radius} $\rad(G)$ is the minimum eccentricity among all vertices of $G$. For a vertex $v\in V$, the $j$-{\em sphere} of $v$ is the set of vertices at distance $j$ to $v$ denoted by $S_j(v,G)$ (i.e., $S_{j}(v,G)=\{u: d(u,v)=j, u\in V\}$).
The {\em transmission} of a vertex $v$ is denoted by $\sigma_G(v),$ or $\sigma(v)$ if the graph $G$ is clear, and equals the sum of distances towards all other vertices:
$\sigma(v) = \sum_{u \in V} d(v,u).$

Let $G=(V,E)$ be a graph with vertex set $\{v_1,v_2,\ldots,v_n\}$.
Dehmer~\cite{Dehmerfunctionals} defined a general form of graph entropy for $G$ using an information functional $f(v_i)$ by the formula
\begin{align*}
I_{f}(G)=-\sum_{i=1}^{n}\frac{f(v_i)}{\sum_{j=1}^{n}f(v_j)}\log\left(\frac{f(v_i)}{\sum_{j=1}^{n}f(v_j)}\right).
\end{align*}
An entropy of $G$ 
taking into account all $j$-spheres
is defined in \cite{Dehmer_Kraus,KDM13}
by the formula
\begin{align*}
I_{f_{s}}(G)=-\sum_{i=1}^{n}\frac{f_{s}(v_i)}{\sum_{j=1}^{n}f_{s}(v_j)}\log \left(\frac{f_{s}(v_i)}{\sum_{j=1}^{n}f_{s}(v_j)}\right),
\end{align*}
where $f_{s}(v_i)={\sum_{j=1}^{\diam(G)}c_j|S_j(v_i,G)|}$, $c_j>0$ and $1\leq j\leq \diam(G)$. 
If we relax the condition of $c_j$ from greater than zero to greater than or equal to zero, then for $c_1=1$ and $c_j=0$, $j\ge 2$, this leads to the first degree-based entropy $I_d(G)$ defined in \cite{CaoMR3212009}. For a graph $G$ with degree sequence $(d_i)_{1 \le i \le n}$ and size $m$, the formula for $I_d(G)$ is
$$I_d(G)=-\sum_{i=1}^n \frac{d_i}{2m}\log\left( \frac{d_i}{2m} \right).$$
When $c_j=j$ for every $1 \le j \le \diam(G)$ (which is a monotonic sequence, the latter being a necessary condition to be able to distinguish based on distance-related properties), we obtain a graph entropy defined before in~\cite{AL11}, which we will refer to as the {\em Wiener-entropy.
The latter is defined by the formula
\begin{align*}
I_{w}(G)=-\sum_{i=1}^{n}\frac{\sigma(v_i)}{\sum_{j=1}^{n}\sigma(v_j)}\log\left(\frac{\sigma(v_i)}{\sum_{j=1}^{n}\sigma(v_j)}\right),
\end{align*}
where $\sigma(v_i)$ is the transmission of $v_i$. 
The {\em Wiener index} of $G$, proposed by Wiener \cite{Wiener1193a005}, is defined by
\begin{align*}
W(G)=\sum\limits_{\{v_i,v_j\}\subseteq V} d(v_i,v_j).
\end{align*}
Since $\sum_{j=1}^{n}\sigma(v_j)=2W(G)$, we have
\begin{align*}
I_{w}(G)=\log(2W(G))-\frac{1}{2W(G)}\sum_{i=1}^{n}\sigma(v_i)\log(\sigma(v_i)).
\end{align*}

In \cite{Dehmer_Kraus,KDM13}, they also defined the entropy of $G$ regarding the eccentricity by
\begin{align*}
I_{f_{e}}(G)=-\sum_{i=1}^{n}\frac{f_{e}(v_i)}{\sum_{j=1}^{n}f_{e}(v_j)}\log \left(\frac{f_{e}(v_i)}{\sum_{j=1}^{n}f_{e}(v_j)}\right),
\end{align*}
where $f_{e}(v_i)=c_iecc(v_i)$, $c_i>0,$ $1 \leq i\leq n$. 
For $c_i=1$,
the {\em eccentricity-entropy} is defined in \cite{Dehmer_Kraus,KDM13} by the formula
\begin{align*}
I_{ecc}(G)=-\sum_{i=1}^{n}\frac{ecc(v_i)}{\sum_{j=1}^{n}ecc(v_j)}\log\left(\frac{ecc(v_i)}{\sum_{j=1}^{n}ecc(v_j)}\right).
\end{align*}



\subsection{Maximum graph entropies}\label{subsec:maxI_G}

The following is well-known, but we state it in general for clarity.

\begin{thr}
    Let $G=(V,E)$ be a graph, and let $f:V \rightarrow \mathbb{R}^{+}$ be an information functional.
    Then $I_f(G) \le \log(n),$ with equality if and only if $G$ is $f$-regular in the sense that $f(v)$ is a constant for every $v \in V.$
\end{thr}

\begin{proof}
    In general, the Shannon entropy of a discrete variable $X$ with set of possible outcomes $\{x_1,x_2,\ldots,x_n\}$ is upper bounded by $\log(n)$, with equality if and only if $X$ has a uniform distribution.
    This is a well-known fact and consequence of Jensen's inequality (since $f(x)=-x \log x$ is concave on $[0,1]$).
    As a corollary, this implies that $I_f(G) \le \log(n),$ with equality if and only if $f(v)$ is a constant for every $v \in V.$ 
\end{proof}

As a corollary, we have the following cases for the main examples of degree- and distance-based graph entropies.

\begin{cor}
    For a graph $G$ of order $n$, 
    \begin{itemize}
        \item $I_d(G) \le \log (n)$, with equality if and only if $G$ is a regular graph,
        \item $I_w(G) \le \log(n),$ with equality if and only if $G$ is a transmission-regular graph,
        \item $I_{\ecc}(G) \le \log(n),$ with equality if and only if $G$ is a self-centered graph. 
    \end{itemize}
\end{cor}

A graph is {\em transmission-regular} if $\sigma(v)=\sigma(u)$ for every $u,v \in V.$ Such graphs have been studied in the past, also under the name distance-balanced graphs, see e.g.~\cite{Abiad_etal17,Bala_etal09,Handa99}. If all the vertices in $G$ have the same eccentricity,
then $G$ is a {\em self-centered} graph \cite{MR1110802}.
If one is restricting the class to trees, determining the maximum given the order is not trivial anymore.
For $I_d$, it is known that the path is extremal, see e.g.~\cite[Thm. 1]{CaoMR3212009} and \cite[Prop.7]{CM22+}.
In Section~\ref{sec:Iecc} we prove that the trees maximizing $I_{\ecc}$ among all trees of order $n \ge 4$ are precisely the trees with diameter $3.$
For $I_w$, the extremal tree is different. Whenever $n \ge 5$ and $T$ is a tree of diameter $3$, $I_w(T)<I_w(S_n)$.
While in the case of $S_n$, all transmissions except from one are equal, it is not extremal for large $n$.
For example when $n=88$, the double broom of diameter $11$ ($P_{10}$ with both endvertices connected with $39$ different pendent vertices) has a larger Wiener-entropy.
The intuition for this counterexample to the potential extreml graph $S_n$ is that all leaves are in the same orbit, and also the transmission of the non-leaf vertices (while there are multiple) are not so different.


\subsection{Contributions}\label{subsec:contrib}

When one considers graphs, eccentricity and transmission are the local analogues of diameter and total distance (linearly related to average distance), the two main distance measures for graphs.
As such, among distance-based entropies, the eccentricity-entropy and Wiener-entropy are among the most natural ones (besides the more general versions).
In this paper, we focus on these two elementary graph entropies, to present the core intuition and methods that are useful when attacking similar questions for other graph entropies. We do this while solving three conjectures posed by Dehmer, Kraus and Schutte~\cite{Dehmer_Kraus,KDM13}.

Before that, in Section~\ref{sec:gen}, we collect and prove some general results and observations about the Shannon entropy of normalized sequences (the probability distribution linearly related with the sequence). These general results are useful to apply and give intuition on questions for graph entropies, but might also be handy when working on general problems about (Shannon) entropy.
In Section~\ref{sec:Iecc}, we focus on the eccentricity-entropy and solve the following three conjectures.
The first one, \cite[Conj.~6.2]{KDM13}, is proven to be true, except from the statement about the removal of a few edges. 
Conjecture~\ref{conj_tree_max_diam3} (\cite[Conj. 4.6]{Dehmer_Kraus}) is true for $I_{\ecc}$ but false for general $I_{f_{e}}$ since one can choose the $c_i$ in such a way that ones favourite tree is extremal. 
Finally, Conjecture~\ref{conj_tree_max_wrong} (\cite[Conj. 4.3]{Dehmer_Kraus}) is disproved.

\begin{conj}[\cite{KDM13}]\label{conj_graph_min_diam2}
Among graphs of order $n$, the minimum value of $I_{ecc}$ is attained by the graph obtained by removing a small number of edges from the complete graph of order $n$. In particular, extremal graphs of order $n$ will have $k\geq \frac{n}{2}$ vertices of degree $n-1$.
\end{conj}

\begin{conj}[\cite{Dehmer_Kraus}]\label{conj_tree_max_diam3}
Among trees of order $n$, the maximum value of $I_{\ecc}$ is attained by the tree obtained by attaching $n-3$ vertices to an end vertex of the path of length $2$.
\end{conj}

\begin{conj}[\cite{Dehmer_Kraus}]\label{conj_tree_max_wrong}
Among trees of order $n$ and diameter $d\ll n$, the maximum value of $I_{f_{s}}$ and $I_{f_{e}}$ are attained by the tree obtained by identifying the central vertex of the star of order $n-d$ with a central vertex of the path of length $d$.
\end{conj}

The proofs mainly rely on the general results in Section~\ref{sec:gen}.
Intuitively, the eccentricity sequence has to be as unbalanced or balanced as possible to attain the minimum or maximum value for the entropy. 

In Section~\ref{sec:asminIw}, we prove that the minimum of $I_w(G)$ among all graphs of order $n$ is of the form $\left(\frac 34 +o(1) \right) \log(n).$
Since the minimum of $I_{\ecc}(G)$ is  $(1-o(1))\log(n)$,
the Wiener-entropy has a better distinguishing character, i.e., the difference between the maximum and minimum is larger.

Finally, in Section~\ref{minIw_extrG_ft} we give some remarks on the trees and graphs that conjecturally attain the minimum Wiener-entropy. 
For this, we define the graph $G_{n,k,j}$ formally as follows.
Take the disjoint union of a path $P_k$ and clique $K_{n-k}$, and connect one end vertex of the path with $j$ vertices of the clique. An example of such a graph $G_{n,k,j}$ has been presented in Figure~\ref{fig:graphGnkj}.

\begin{figure}[ht]
\centering
\begin{tikzpicture}
\foreach \x in {0,45,...,315}{\draw[fill] (\x+22.5:1.2) circle (0.05);

\draw(\x+22.5:1.2) -- (\x-22.5:1.2);}

\foreach \x in {-90,-45,0,45}{\draw[fill] (\x+22.5:1.2) circle (0.05);
\draw(\x+22.5:1.2) -- (0:3);
}

\foreach \x in {0,45,90,135}{\draw[thick] (\x+22.5:1.2) -- (\x+202.5:1.2);}
\draw[thick] (22.5:1.2) -- (112.5:1.2) -- (202.5:1.2) -- (292.5:1.2) -- cycle;
\draw[thick] (67.5:1.2) -- (157.5:1.2) -- (247.5:1.2) -- (337.5:1.2) -- cycle;

\draw[thick] (22.5:1.2) -- (157.5:1.2) -- (292.5:1.2) -- (67.5:1.2) -- (202.5:1.2) -- (337.5:1.2) -- (112.5:1.2) -- (247.5:1.2) -- cycle;
\foreach \x in {3,...,8}{\draw[fill] (0:\x) circle (0.05);}
\draw[thick] (0:3) -- (0:8);
\end{tikzpicture}
\caption{The graph $G_{n,k,j}$ for $n=14, k=6$ and $j=4$}\label{fig:graphGnkj}
\end{figure}
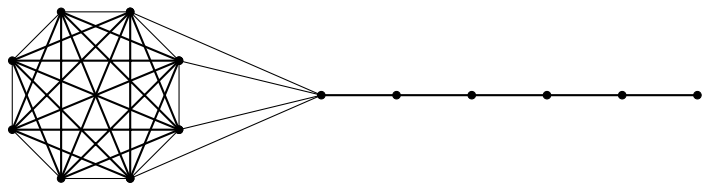

It turns out that already within the subclass of graphs of the form $G_{n,k,j}$, there is some surprising behaviour.
When $n$ ranges from $16$ to $100$, the value of $j$ fluctuates a lot, which intuitively can be explained as the clique (colex graph) at the end growing in a less discretized way.
Nevertheless, for $n$ large, it seems that $j=1$ is always true, i.e., we conjecture that $I_w$ for graphs of order $n$ is maximized by a graph of the form $G_{n,k,1}$ whenever $n$ is sufficiently large. 
Explaining this evolution in behaviour of the graphs $G_{n,k,j}$ seems to be already interesting.

\section{General results on entropy}\label{sec:gen}

The {\em normalized vector} (or {\em unit vector}) of a non-zero vector $\mathbf{a}=(a_1, a_2, \ldots, a_n)$, is a vector in the same direction with norm $1$. It is denoted by $\widehat{\mathbf{a}}$ and given by $\widehat{\mathbf{a}}=\frac{\mathbf{a}}{\abs{\mathbf{a}}_1}$, where $\abs{\mathbf{a}}_1=\sum_{i=1}^n \abs{a_i}$ is the $1$-norm of $\mathbf{a}$. 
By $\floorceil{\b}$ we denote a number that is either the greatest integer less than or equal to $\b$, $\bfloor \b$, or the smallest integer greater than or equal to $\b$, $\bceil \b$.

We also highlight the following two definitions.

\begin{defi}
Let $\mathbf{a}=(a_1,a_2,\ldots,a_n)$ be a positive real vector, and let $\mathbf{p}=\widehat{\mathbf{a}}=(p_1,p_2,\ldots,p_n)$. We define the entropy of $\mathbf{a}$ by
\begin{align*}
H(\mathbf{a})&=-(p_1,p_2,\ldots,p_n)(\log(p_1),\log(p_2),\ldots,\log(p_n))^{\mathrm{T}}\\
                 &=-\sum_{j=1}^{n} p_j\log(p_j).
\end{align*}
\end{defi}

\begin{defi}
    We say a sequence $\mathbf{a}=(a_i)_{1 \le i \le n}$ majorizes $\mathbf{b}=(b_i)_{1 \le i \le n}$
    if for every $1 \le k \le n,$ the sum of the $k$ largest (resp., smallest) elements of $\mathbf{a}$ is at least (resp., at most) the sum of the $k$ largest (resp., smallest) elements of $\mathbf{b}$, and equality does hold when $k=n.$
\end{defi}

The following theorem is a direct consequence of Karamata's inequality~\cite{Karamata32} applied on the (strictly) concave function $-x \log(x)$ (for $x \in [0,1]$).

\begin{thr}\label{thr:majorized_Hineq}
    Let $\mathbf{a}=(a_i)_{1 \le i \le n}$ and $\mathbf{b}=(b_i)_{1 \le i \le n}$ be two different sequences of positive reals such that $\mathbf{a}$ majorizes $\mathbf{b}.$
    Then  $$H(\mathbf{a})<H(\mathbf{b}).$$
\end{thr}

Bauer's maximum principle~\cite{Bauer58} applied on the concave function $H$ implies the following proposition (proof can be skipped by readers familiar with the principle).

\begin{prop}\label{prop:concavity_entropy}
    Let $S \subset (\mathbb R^+)^n$ be the convex hull of $N$ points.
    If $H(\mathbf{s})$ attains the minimum value among $\mathbf{s} \in S$, then $\mathbf{s}$ is one of the extremal points of the convex hull.
\end{prop}

\begin{proof}
    Assume this is not the case. So 
    $\mathbf{s}$ is not one of the extremal points of the convex hull.
    In that case there is a non-zero 
    vector $\mathbf{t}$ for which the interval $\left[\mathbf{s-t}, \mathbf{s+t}\right]$ is fully contained in $S$.
    Since $(\mathbf{s-t})+( \mathbf{s+t})= 2 \mathbf{s}$, we have
    $\abs{ \mathbf{s-t} }_1 \widehat{ \mathbf{s-t} }   + \abs{\mathbf{s+t} }_1 \widehat{ \mathbf{s+t} } = 2 \abs{\mathbf{s} }_1 \widehat{ \mathbf{s} }.$
    So $ \widehat{\mathbf{s}} = \lambda \widehat{\mathbf{s-t}}  + (1-\lambda) \widehat{\mathbf{s+t}},$ where $0<\lambda<1$.
    
    Since the entropy is strictly concave, this immediately implies that 
    $$\min\{H(\mathbf{s-t}), H(\mathbf{s+t})\} <H(\mathbf{s}).$$
    This leads to a contradiction with the choice of $\mathbf{s}.$
\end{proof}

\begin{prop}\label{prop:fix_max_entropy}
    Let $a_1,a_2, \ldots, a_s$ be fixed positive reals and $b_1,b_2, \ldots, b_t$ be positive variables.
    Let $\b$ be the solution for $\log(\b) = \frac{ \sum_{j=1}^s a_j \log(a_j) }{ \sum_{j=1}^s a_j }.$
    Then the entropy of (the normalized vector of) $(a_1,a_2, \ldots, a_s,b_1, \ldots, b_t)$ is maximized if and only if 
    $b_1=b_2=\cdots =b_t=\b.$
    Furthermore, $H\left((a_1,a_2, \ldots, a_s,\underbrace{b, \ldots, b}_{t }) \right)$ is increasing when $b<\b$ and decreasing when $b>\b.$
\end{prop} 

\begin{proof}
    First note that due to concavity of $f(x)=-x \log(x),$ once $\sum b_i $ is fixed, we know that the maximum occurs when all $b_i$ are equal to a value $b$.
    In that case $$H((a_1,a_2, \ldots, a_s,b_1, \ldots, b_t)) = \log\left(\sum a_i + bt\right) - \frac{ \sum a_i \log(a_i) + tb \log(b)}{\sum a_i + bt}.$$
    The maximum can be found by taking the derivative towards $b$ and setting this equal to zero:

   \begin{align*}
       \frac{d}{db}H\left((a_1,a_2, \ldots, a_s,\underbrace{b, \ldots, b}_{t })\right) &=
       \frac{t}{ \sum a_i + bt}  - \frac{ t(1+ \log(b) )}{\sum a_i + bt} + t \frac{ \sum a_i \log(a_i) + tb \log(b)}{(\sum a_i + bt)^2}\\
       &= \frac{  \sum a_i \log(a_i) -  \log(b)\sum a_i}{(\sum a_i + bt)^2}.
   \end{align*}
    
    So this is zero when $\log(b) = \frac{ \sum_{j=1}^s a_j \log(a_j) }{ \sum_{j=1}^s a_j },$ i.e., $b=\b$, and it is positive respectively negative if $b$ is smaller or larger,
    implying it indeed attains a maximum for this choice.
\end{proof}

\begin{prop}\label{prop:fix_max_entropy_integers}
    Let $a_1,a_2, \ldots, a_s>1$ be fixed reals and $b_1,b_2, \ldots, b_t$ be variables that are positive integers.
    Let $\b$ be the real solution for the equality $\log(\b) = \frac{ \sum_{j=1}^s a_j \log(a_j) }{ \sum_{j=1}^s a_j }.$
    Then the entropy of $(a_1,a_2, \ldots, a_s,b_1, \ldots, b_t)$ is maximized for $b_1=b_2=\cdots=b_t=b$, where $b$ is the optimal choice in $\{ \bfloor{\b}, \bceil{\b} \}.$
\end{prop}

\begin{proof}
Let $(b_i)_{1 \le i \le t}$ be positive integers maximizing $H((a_1,a_2, \ldots, a_s,b_1, \ldots, b_t)).$
As a first step, we prove that all $b_i$ are equal to $\floorceil{\b}.$
Assume, to the contrary, that $b_1 \le b_2 \le \cdots \le b_t$ (that is, the sequence is ordered) and $\max\{b_{u-1},  \bceil{\b}\} < b_u=b_{u+1}= \cdots =b_t$.
(The analogous case, where some of the numbers are smaller than $\bfloor{\b}$, follows similarly.)

Now applying Proposition \ref{prop:fix_max_entropy} with fixed reals $(a_1,a_2, \ldots, a_s,b_1, \ldots, b_{u-1})$ and variables $(t-u+1)$ times a variable $b$, we conclude that 
$H\left((a_1,a_2, \ldots, a_s,b_1, \ldots, b_{u-1},\underbrace{b, \ldots, b}_{t-u+1 } )\right)$ is decreasing when $b>\b'.$
Here $\b'\le \max\{\b, b_{u-1}\}\le b_u-1.$
Hence decreasing $b_u=b_{u+1}= \cdots =b_t$ by one increases the entropy. This contradiction implies the result.
Alternatively, one can repeat the above argument of increasing the entropy by decreasing all the occurrences of the largest value by one if it is greater than $\bceil{\b}$. This implies that $\max\{b_i\}\le \bceil{\b}$ for an optimal choice of $(b_i)_{1 \le i \le t}.$
Analogously an optimal sequence satisfies $\min\{b_i\}\ge \bfloor{\b}.$

In the second step, we prove that all $b_i$ are equal.
Assume this is not the case and $t_1$ of the $b_i$ equal $b_1=\bfloor{\b}$ and $t_2=t-t_1$ equal (possible after renaming the index) $\bceil{\b}=b_2.$
Consider $(a_1,a_2,\ldots, a_{s},b_3,\ldots,b_{t})$ as fixed values and $b_1$ and $b_2$ as the variable ones.
Now there do exist positive reals $c_1, c_2$ for which
\begin{align*}
    c_1 b_1 + c_2 b_2 &= \sum_{i=1}^{s} a_i + \sum_{i=3}^{t} b_i \\
    c_1 b_1\log(b_1) + c_2 b_2\log(b_2) &= \sum_{i=1}^{s} a_i \log(a_i) + \sum_{i=3}^{t} b_i  \log(b_i).
\end{align*}

We will need the following claim in the remaining of the proof.
\begin{claim}\label{claim_convex_function}
    For fixed positive reals $n$ and $b\ge 1$, consider the function 
    $$ h(c)= \log( nb+c) - \frac{ c(b+1)\log(b+1) +(n-c) b\log(b)}{nb+c}.$$
    Then $h(c)$ is a strictly convex function on $[0,n].$
\end{claim}

\begin{claimproof}
Since $h(c)$ is a function that is twice continuously differentiable on $[0,n]$, to prove $h(c)$ is strictly convex, it is sufficient to prove its second derivative is positive for all $c$ in $[0,n]$. By calculating, its second derivative is $\frac{2b(b+1)n(\ln(b+1)-\ln(b))-bn-c}{(bn+c)^3\ln(2)}.$ 
This second derivative is positive since $bn+c \le (b+1)n$ and $2b(\ln(b+1)-\ln(b))>1$ for every $ b \ge 1$.
\end{claimproof}

We apply Claim~\ref{claim_convex_function} with fixed $b=\bfloor{\b}$ and $n=c_1+c_2+2$. Let $c=c_2+1.$ Then due to convexity of $h$, $h(c)<\max\{h(c-1),h(c+1)\}.$
But this is exactly telling that changing $b_1$ into $\bceil{\b}$ or $b_2$ into $\bfloor{\b}$ will increase the entropy $H((a_1,a_2, \ldots, a_s,b_1, \ldots, b_t)).$
This contradicts the choice of the sequence having the maximum entropy and thus we can conclude that all $b_i$ are equal. 
\end{proof}

\begin{defi}
    Let $H^n((a_1,a_2, \ldots, a_s))=H\left((a_1,a_2, \ldots, a_s,\underbrace{b, \ldots, b}_{n-s} )\right)$ where $b$ is chosen such that $\log(b) = \frac{ \sum_{j=1}^s a_j \log(a_j) }{ \sum_{j=1}^s a_j }.$
\end{defi}

\begin{lem}\label{lem:log(n-r)}
    $H^n((a_1,a_2, \ldots, a_s))= \log(n-r)$ where $r= s- \frac{ \sum_{j=1}^s a_j}{b}.$
\end{lem}
    
\begin{proof}
    Note that 
    \begin{align*}
        H\left((a_1,a_2, \ldots, a_s,\underbrace{b, \ldots, b}_{n-s}) \right)&=
        \log\left( \sum_{j=1}^s a_j + (n-s) b\right) - \frac{ \sum_{j=1}^s a_j \log(a_j) + (n-s) b \log(b)}{\sum_{j=1}^s a_j + (n-s) b}\\
        &= \log\left( \sum_{j=1}^s a_j + (n-s) b\right) - \log(b)\\
        &= \log\left(n-s+\frac{\sum_{j=1}^s a_j}{b} \right).
    \end{align*}
\end{proof}

\begin{prop}\label{prop:majorizesHineq}
    Let $\mathbf{a}=(a_i)_{1 \le i \le s}$ and $\mathbf{c}=(c_i)_{1 \le i \le s}$ be two sequences of positive reals such that $\mathbf{a}$ majorizes $\mathbf{c}.$
    Then for every $n \ge s,$ $$H^n(\mathbf{a})\le H^n(\mathbf{c}).$$
\end{prop}

\begin{proof}
    Let $\log(b) = \frac{ \sum_{j=1}^s a_j \log(a_j) }{ \sum_{j=1}^s a_j }$ and $\log(b') = \frac{ \sum_{j=1}^s c_j \log(c_j) }{ \sum_{j=1}^s c_j }.$
    By Theorem~\ref{thr:majorized_Hineq} and Proposition~\ref{prop:fix_max_entropy} (with $b'$ as maximizer) respectively, $$H\left((a_1,a_2, \ldots, a_s,\underbrace{b, \ldots, b}_{n-s} )\right) \le H\left((c_1,c_2, \ldots, c_s,\underbrace{b, \ldots, b}_{n-s}) \right) \le H\left((c_1,c_2, \ldots, c_s,\underbrace{b', \ldots, b'}_{n-s} )\right).$$
\end{proof}

\begin{remark}\label{rem:H^n_larger_subset}
    Observe that by definition, if $\mathbf{c}$ is a subsequence of $\mathbf{a},$ then $H^n(\mathbf{a})\le H^n(\mathbf{c}).$
\end{remark}

\section{On eccentricity-entropy}\label{sec:Iecc}

This section is devoted to the three conjectures mentioned in Section~\ref{sec:int}. We start proving Conjecture~\ref{conj_graph_min_diam2}.

\begin{thr}
    The minimum value for $I_{\ecc}$ among all graphs of order $n$ is attained by graphs of diameter $2$.
\end{thr}

\begin{proof}
Let $G$ be a graph attaining the minimum value for $I_{\ecc}(G)$ among all graphs of order $n$. Let $r$ be the radius of the graph and let $S= [r,2r]^n.$ The sequence of eccentricities $\left(\ecc(v) \right)_{v \in V}$ belongs to $S$ and by Proposition~\ref{prop:concavity_entropy} 
every eccentricity is $r$ or $2r$ if it attains the minimum in $S$.
If all eccentricities are equal, then the entropy equals $\log(n)$ and so it would be the maximum instead of minimum.
If $r>1$, then if there are vertices with eccentricity $r$, as well as with $2r$, there is also a vertex with eccentricity $2r-1.$
When $r=1$, equality can clearly be attained and since the normalization made the precise value of $r$ being unimportant, we conclude that the graphs with radius $1$ indeed attain the minimum.
\end{proof}

\begin{remark}
    A vertex has eccentricity $1$ if and only if it has degree $n-1.$
    The entropy is completely determined once the number $k$ of vertices of degree $n-1$ is known. The other $n-k$ vertices have eccentricity $2.$
    The entropy of the graph equals $\log(2n-k)-\frac{2(n-k)}{2n-k}.$
    This is a concave function for $0 \le k \le n$, with the minimum being attained by $k=(2-2\ln(2)) n$ (over the reals, and so $k $ will be of the form $\floorceil {(2-2\ln(2)) n}$). This minimum is roughly $\log(n)-0.086=(1-o(1))\log_2(n).$
    As such, for $n \ge 10$, we immediately have that $k > \frac n2,$ so together with the verification for small $n$, this addresses Conjecture \ref{conj_graph_min_diam2} 
    completely.
    Since there is a set $S$ of $s=\floorceil{(2\ln(2)-1) n}$ vertices with eccentricity $2$, the complement of $G[S]$ needs to have degree at least $1$.
    In particular, taking $G[S]$ to be the empty graph, we observe that there are extremal graphs with at least $ \binom{\bfloor{(2\ln(2)-1) n}}{2}$ non-edges. It also implies that there are many minimal graphs, $2^{\Theta(n^2)}$ (roughly the number of non-isomorphic connected graphs on $s$ vertices), contrasting the ideas of~\cite{KDM13}.
\end{remark}

When restricting to the class of trees, we will observe that the star (the only tree with diameter $2$) is not the graph minimizing the eccentricity-entropy, mainly due to the reason that there is no possibility to play with the ratio of eccentricities with values $1$ and $2$. Actually the trees minimizing $I_{\ecc}$ will be caterpillars containing a central path with many pendent leaves attached to both the central vertices and the end vertices, such that most eccentricities are $r+1$ and $2r$, for some value of $r(n)$. 
On the other hand, the star has the third largest possible value for $I_{\ecc}$ among all trees of order $n$. This will be verified by deriving the three largest possible values of the eccentricity-entropy for the class of trees, and as such we also confirm Conjecture~\ref{conj_tree_max_diam3}.
To do so, we first observe that the eccentricities of vertices on the diameter (the path between two furthest vertices) are fixed.

\begin{remark}\label{rem:fixedecc_onpath}
    Let $P$ be a diametrical path of length $d$ in a tree. If $v$ is a vertex on the path at distance $i$ from one end vertex of $P$, then $\ecc(v)=\max\{i,d-i\}$. Otherwise, there exists a path of length greater than $d$.
    In a tree, there are only one (for even diameter) or two (for odd diameter) vertices whose eccentricity equals the radius $r.$
\end{remark}

\begin{lem}\label{lem:Iecc(Sn)}
    For a star $S_n$, $I_{\ecc}(S_n)>\log\left(n-\frac {1-\ln(2)}{2}\right).$
\end{lem}
\begin{proof}
    Observe that 
    $\frac{d}{dx} \log(x) = \frac{1}{ \ln(2) x}$ and thus
\begin{equation*}
    I_{\ecc}(S_n)=\log\left(n-\frac 12\right) + \frac{\frac 12 }{n-\frac 12}\ge 
        \log\left(n-\frac 12\right) + \frac1{\ln(2)} \int_{n-\frac 12}^{n-\frac {1-\ln(2)}{2} } 
        \frac{1}{x} \d x= \log\left(n-\frac {1-\ln(2)}{2}\right)
\end{equation*} 
\end{proof}

\begin{prop}\label{prop:over6_less_star}
If $T$ is a graph of order $n$ and diameter $d \ge 6$, then $I_{\ecc}(T)<I_{\ecc}(S_n).$ 
\end{prop}

\begin{proof}
For even diameter $d=2r$ where $r \ge 3$ ($r$ is the radius) and $n \ge 2r+1,$ we have by respectively Remark~\ref{rem:H^n_larger_subset} (on subsequences), Proposition~\ref{prop:majorizesHineq} (on majorizing sequences) and Lemma~\ref{lem:log(n-r)} for the final inequality that
\begin{align*}
        H^n((r, r+1,r+1,r+2,r+2,\ldots, 2r-1,2r-1,2r,2r)) &\le H^n((r, r+1,r+1,2r-1,2r-1,2r,2r))\\ 
        &\le H^n((r, \frac 43 r,\frac 43 r,\frac 53 r,\frac 53 r,2r,2r))\\ 
        &= H^n((3,4,4,5,5,6,6))\\
        &<\log(n-0.17).
    \end{align*}
    For the second inequality, it is sufficient to note that $r+1 < \frac 43 r < \frac 53 r<2r-1 $ when $r \ge 4$ to deduce the majorization.
    Analogously for $d=2r-1$ odd and $n \ge 2d$ we have $$H^n((r,r, r+1,r+1,2r-2,2r-2,2r-1,2r-1)) \le H^n((4,4,5,5,6,6,7,7))<\log(n-0.16).$$
    Since $\frac {1-\ln(2)}{2}<0.154,$ we conclude by~\cref{lem:Iecc(Sn)} and Remark~\ref{rem:fixedecc_onpath}, the latter stating that the distances $r+1,r+2, \ldots, 2r-1,2r$ (each with multiplicity two) and $r$ all appear in the eccentricity sequence of the graph.
\end{proof}

\begin{prop}\label{prop:diam45_less_star}
    If a tree $T$ of order $n$ with diameter $5$ satisfies $I_{\ecc}(T)>I_{\ecc}(S_n)$, then it has at most $2$ vertices with eccentricity $5.$
    Every tree $T$ of order $n$ with diameter $4$ satisfies $I_{\ecc}(T)<I_{\ecc}(S_n)$
\end{prop}

\begin{proof}

    For diameter $5$,~\cref{rem:fixedecc_onpath} says that the eccentricity sequence of the graph at least contains each of $3,4,5$ with multiplicity no less than $2.$
    By Lemma~\ref{lem:Iecc(Sn)} it is sufficient to compute with Lemma~\ref{lem:log(n-r)} that\\ $H^n\left( (3,3,4,4,5,5,5,5) \right)< \log(n-0.157)$
     for every $n\geq 8$, and $H^n\left((3,3, \underbrace{4, \ldots, 4}_{40},5,5,5) \right)<\log(n-0.15345)$
     for every $n\geq 45$, and verify that
    $H\left((3,3, \underbrace{4, \ldots, 4}_{n-5},5,5,5) \right)<I_{\ecc}(S_n)$ for every $n \le 44.$

    For diameter $4$, similarly it is sufficient to compute that $H^n\left( (2,3,3,4,4,4) \right)< \log(n-0.18)$,\\$H^n\left((2,\underbrace{3, \ldots, 3}_{10},4,4) \right)<\log(n-0.154)$ (for $n$ sufficiently large)
    and $H\left((2,\underbrace{3, \ldots, 3}_{n-3},4,4) \right)<I_{\ecc}(S_n)$ for every $n \le 12.$ 
\end{proof}

By Propositions~\ref{prop:over6_less_star}~and~\ref{prop:diam45_less_star} there are only $3$ candidates for
extremal graphs.

The next result identifies the three largest eccentricity-entropy values for trees of a given order, along with their corresponding trees or eccentricity sequences.

\begin{prop}
    Among all trees of order $n$, the three largest possible values of $I_{\ecc}$ are obtained (in order) by a tree $T_3$
     of diameter $3$, a tree $T_5$ with
        eccentricity
        sequence $(3,3,\underbrace{4,\ldots,4}_{n-4},5,5)$, and the star $S_n$
\end{prop}

\begin{proof}
We compute that (computations analogous to the ones in the proof of Lemma~\ref{lem:Iecc(Sn)})
\begin{align*}
    I_{ecc}(T_3)-I_{ecc}(T_5)&=\log(n-2/3)+\frac{4\log(3)-4\log(2)}{3n-2}-\log(n)-\frac{16-3\log(3)-5\log(5)}{2n}\\
    &\ge \log(n-2/3)-\log(n)+\frac{(4/3+3/2)\log(3) +5/2 \log(5) - (8+4/3)  }{n-2/3}>0
\end{align*}
and 
\begin{align*}
I_{ecc}(T_5) -I_{\ecc}(S_n)&=\log(n)+\frac{16-3\log(3)-5\log(5)}{2n}-\log\left(n-\frac{1}{2}\right)-\frac{1}{2n-1}\end{align*} is a strictly decreasing function (for $n\ge 6$) for which the limit is zero, so always positive.
\end{proof}

To end this section, we consider trees of given order and diameter to disprove Conjecture \ref{conj_tree_max_wrong}.
Note that~\cite[Rem.~4.4]{Dehmer_Kraus} already observed that the condition $d\ll n$ was necessary and $d$ needed to be fairly small for fixed order $n$.
Our result determines the eccentricity sequence of the maximizing trees exactly (up to the determination of $b$).

\begin{prop}\label{prop:tree_max_entropy_0.764d}
There exists a value $b$ such that the maximum value for $I_{{ecc}}$ among all trees of diameter $d$ and order $n$ is obtained by the trees with eccentricity sequence
\begin{itemize}
  \item $(\frac{d}{2},\frac{d}{2}+1,\frac{d}{2}+1,\ldots,d,d,\underbrace{b,\ldots,b}_{n-d-1})$ 
      for even $d$
  \item $(\frac{d+1}{2},\frac{d+1}{2},\ldots,d,d,\underbrace{b,\ldots,b}_{n-d-1})$ 
      for odd $d.$
\end{itemize}
Here $b\sim\frac{\sqrt[3]{2}}{\sqrt{e}}d$ as $ d \to \infty.$
\end{prop}

\begin{proof}

Let $T$ be a tree attaining the maximum value for $I_{ecc}$ among trees of order $n$. 
Since we can attach pendent leaves with any eccentricity in the range $[\rad(T)+1, \diam(T)]$ to a diametral path, we conclude from Remark~\ref{rem:fixedecc_onpath} (which reminds us of the unavoidable eccentricities) and Proposition~\ref{prop:fix_max_entropy_integers} (telling that the other eccentricities are all equal in the extremal case).

For $d$ sufficiently large, we have the following computations, which we compute exactly up to $1+O\left( \frac 1d \right)$ factor.
\begin{align*}
    \log(b) &\sim \frac{2\sum_{i=\frac{d}{2}}^{d}i\log (i)}{2\sum_{i=\frac{d}{2}}^{d}i}\\
    &\sim  \frac{\int_{\frac{d}{2}}^{d} x log (x) \d x}{\int_{\frac{d}{2}}^{d} x \d x}\\
    &\sim \frac{ \frac 38 d^2 \log (d) - \frac 3 {16\ln(2)} d^2 +\frac{1}{8}d^2}{\frac 38 d^2} \\
    &= \log (d) - \frac {1}{2\ln(2)} +\frac{1}{3}.\\
\end{align*}
Since $\exp( log(d)/d)=1+o(1), $ we conclude that 
\begin{align*}
b&\sim  \frac{\sqrt[3]{2}}{\sqrt{e}}d \sim 0.764 d.
\end{align*}
\end{proof}

\begin{cor}
    Since the eccentricity-entropy $I_{ecc}$ is an example for $I_{f_{e}},$
    Conjecture \ref{conj_tree_max_wrong} is not true when $d$ is sufficiently large, as then the value $b$ in Proposition~\ref{prop:tree_max_entropy_0.764d} satisfies $b> \bceil{\frac{d}{2}}+1.$
\end{cor}

\section{Asymptotic minimum Wiener-entropy of graphs}\label{sec:asminIw} 

We start proving the following two elementary lemmas.
\begin{lem}\label{lem:sigmavsW}
    Let $G$ be any connected graph. Then for every vertex $v,$ $(n-1)\sigma(v) \ge W(G).$
    Equality holds if and only if $G$ is a star and $v$ is its center. 
\end{lem}

\begin{proof}
    By rewriting the sum and applying the triangle-inequality,
    \begin{align*}
        (n-1)\sigma(v) &=\sigma(v)+(n-2)\sum_{u \in V \backslash v} d(v,u)\\
        &=\sigma(v)+\sum_{ \{w, u\} \subset V \backslash v} \left(d(u,v)+d(v,w)\right)\\
        &\ge \sigma(v)+\sum_{\{w, u\} \subset V \backslash v} d(u,w)\\
        &=W(G).
    \end{align*}
    Equality only appears if for every $\{w, u\} \subset V \backslash v$, there is a shortest path from $u$ to $w$ containing $v$. In particular $d(u,w) \ge 2$ for all $\{w, u\} \subset V \backslash v$ implies that $V\backslash v$ is an independent set. Since $G$ is a connected graph, this implies that $G$ is a star with center $v$.
\end{proof}

\begin{lem}\label{lem:diffsigma}
    Let $G$ be any connected graph. 
    Let $uv \in E(G)$ be an edge.
    Then $\sigma(v) \ge n-1.$
    Also $\abs{\sigma(u)-\sigma(v)} \le n-2, $ with equality if and only if $u$ or $v$ is a pendent vertex.
\end{lem}

\begin{proof}
    The first observation is trivial, since $\sigma(v)$ is the sum of $n-1$ distances that are all at least one.
    Let $w$ be a vertex different from $u$ and $v$, then $d(u,w)\le d(u,v)+d(v,w)=d(v,w)+1$ and vice versa, so $\abs{d(u,w)-d(v,w)} \le 1.$
    Hence we conclude, by applying the triangle inequality again; $\abs{\sigma(u)-\sigma(v)}=\abs{\sum_{w \in V \backslash \{u,v\} } \left(d(u,w)-d(v,w)\right) } \le \sum_{w\in V\backslash\{u,v\}} |d(u,w)-d(v,w)|
    \le n-2. $
\end{proof}

\begin{prop}\label{prop:Iwbroom}
    Let $T$ be the broom consisting of a path $P_k$ with one of its end vertices $c$ connected with $n-k$ pendent vertices.
For fixed $\frac{1}{3}> \epsilon>0$, let $k=n^{1/2+\epsilon}$ and let $n$ be sufficiently large.
Then $I_w(T) \sim \frac{3+2\epsilon}{4} \log(n).$
\end{prop}

\begin{proof}
    Note that 
    \begin{align*}
        W(T)&=W(P_k)+W(S_{n-k+1})+(n-k) \left(\sum_{i=1}^k i\right)\\
        &=\binom{k+1}{3} +(n-k)^2+(n-k)\binom{k+1}{2}\\
        &\sim \frac{nk^2}{2}.
    \end{align*}

    If $\ell$ is one of the end vertices of the star, then $\sigma(\ell)=2(n-k-1)+\sum_{i=1}^k i \sim \frac{k^2}{2}.$

    If $v$ is a vertex on the path at distance $i-1$ from $c$, then $\sigma(v)=i(n-k)+ \binom{i}{2}+\binom{k-i+1}{2}.$
    The sum of the transmissions for the vertices at distance $i-1$ from $c$ for $2\le i \le n^{2\epsilon}$ is bounded by
    $$\sum_{i=2}^{ n^{2\epsilon}} \left(in + \binom{k}{2}\right) < n\frac{n^{4\epsilon}}{2} + n^{2\epsilon} \frac{n^{1+2\epsilon}}{2}=n^{1+4\epsilon}.$$
    
    Hence the sum of associated values $p_i=\frac{\sigma(v)}{2W(G)}$ is of the order
    $\frac{n^{1+4\epsilon}}{nk^2}=O(n^{2\epsilon-1})=o(1).$
    By Lemma~\ref{lem:sigmavsW}, we have that every $p_i$ is at least $\frac{1}{2(n-1)}$ and thus $- \log(p_i) \le \log(2(n-1)).$
    This implies that the contribution to $I_w(G)$ of the vertices at distance $i-1$ from $c$ for $2\le i \le n^{2\epsilon}$ is $o(\log n).$
    
    When $i>n^{2\epsilon}$, then $\sigma(v) \sim in.$
    As such, the probability $p_i =\frac{\sigma(v) }{2W(G)} \sim \frac{i}{k^2}.$
    For an end vertex $\ell$ of the star, we have that the associated probability for the functional $\frac{\sigma(\ell) }{2W(G)} \sim \frac{ \frac{k^2}{2}}{k^2n} \sim \frac{1}{2n}.$
    
    All together, this implies that 
    \begin{align*}
        I_w(G) &\sim \sum_{i=n^{2\epsilon}}^k \frac{i}{k^2} \left(2\log(k)-\log(i)\right) + n \frac{1}{2n} \log(2n)\\
        &\sim \log(k) - \frac{1}{k^2} \sum_{i=2}^k i \log(i)  +\frac 12 \log n\\
        &\sim \frac 12 \log n + \frac 12 \log(k)\\
        &=\frac{3+2\epsilon}{4} \log(n).
    \end{align*}
    
    Here we used that $\int_1^k x \log x \d x \sim \frac{k^2}{2} \log(k).$ 
\end{proof}

\begin{thr}\label{thr:minIw_as}
    Let $G$ be a connected graph of order $n$.
    Then $I_w(G) > \frac 34 (1+o(1))\log (n).$
\end{thr}

\begin{proof}
    Let $p_1, \ldots, p_n$ be the $n$ fractions of the form $\frac{\sigma(v)}{2W(G)}$, ordered in a decreasing order, i.e., $p_i \ge p_{i+1}$ for every $i$.
    By Lemma~\ref{lem:sigmavsW}, we know $p_n \ge \frac{1}{2(n-1)}>\frac 1{2n}.$
    By Lemma~\ref{lem:diffsigma} and $2W(G) \ge n(n-1)$, we also note that $p_i - p_{i+1} < \frac{n-2}{n(n-1)}<\frac 1n$ for every $1 \le i \le n-1.$
    
    Let $k$ be the largest number for which $k^2-k \le n.$
    Then the sequence $(p_1, \ldots, p_n)$ is majorized by the sequence $(a_1, \ldots, a_n)$ with $a_1=\frac{n-k^2+3k-1}{2n}$, $a_i = \frac{2k-2i+1}{2n}$ for $2 \le i \le k$ and $a_i=\frac{1}{2n}$ whenever $n \ge i>k.$
    
    Since the function $f(x)=-x \log(x)$ is concave, by Karamata's inequality we have 
    $$\sum_i f(p_i) \ge \sum_i f(a_i).$$
    Note that one term separately is neglectible, i.e., $f(a_1) = o(\log(n))$ and so we can replace it by $f\left( \frac{2k-1}{2n} \right)$ in the estimation.
    
    Now 
    $$(n-k)f\left( \frac{1}{2n} \right) =\frac{1}{2}(1+o(1)) \log(2n)=\frac{1}{2}(1+o(1)) \log(n)$$
    and
    \begin{align*}
        \sum_{i=1}^{k-1} f\left( \frac{2i-1}{2n} \right) &\ge
        \sum_{i=1}^{k-1} f\left( \frac{i}{n} \right)\\
        &=
        \frac 12 (1+o(1)) \log(n) - \frac{1}{n} \sum_{i=1}^{k-1} i \log(i)\\
        &\sim \frac 12 \log(n) - \frac 1n \int_1^{k-1} x \log x \d x\\
        &\sim  \frac 12 \log(n) -\frac 14 \log n\\
        &= \frac 14 \log(n).
    \end{align*}
    
    Together, this implies that $$\sum_i f(p_i) \ge \frac 34 (1+o(1)) \log (n).$$
\end{proof}

\section{Further thoughts on the extremal graphs for $I_w$}\label{minIw_extrG_ft}

In Section~\ref{sec:asminIw}, we determined the minimum value of $I_w(G)$ among graphs of order $n$ asymptotically.
A precise result, or characterizing the extremal graphs seems to be much harder. In this section, we present some thoughts about the extremal graphs.

From the idea of Theorem~\ref{thr:minIw_as},
for the class of trees, the trees of order $n$ with minimum Wiener-entropy are expected to be brooms for sufficiently large $n$.
If this intuition is true, the extremal broom would be completely determined by the length $k$ of the path in~\cref{prop:Iwbroom}.
If $k=o(\sqrt n)$, then $I_w \sim \log_2(n)$ (the graph is nearly a clique; $n-k$ of the $p_i$ are approximately $\frac{1}{n}).$
Together with~\cref{prop:Iwbroom}, we know the optimal choice is of the form $k(n)=n^{1/2+o(1)}$. Here $k(n)$ can be expected to be a step-wise increasing function, since it only attains integers.

For small $n$, the asymptotic analysis does not give a clear indication of the extremal tree.
In particular for $n \le 16$, the extremal trees are not brooms.
For $3\leq n\leq 18$, the trees with the minimum Wiener-entropy are listed in Table~\ref{Tab:extT_minIw}\footnote{See~\url{https://github.com/yndongmath/wiener-entropy} for verification.}.

Next, we focus on graphs instead of trees. 
As a corollary of Proposition~\ref{prop:fix_max_entropy}, one can easily observe that if there are $2$ vertices $u,v$ for which $\sigma(u)$ and $\sigma(v)$ are small (smaller than the corresponding value $\b$) and $uv$ does not affect $\sigma(w)$ for $w \not \in \{u,v\}$, then $uv$ is always present.
The latter also holds for a set of vertices.
Starting from a broom, by the previous, all edges should be present between the leaves of the star, and we obtain the concatenation of a path and a clique.
It has to be observed that for $n \le 15$, the extremal graphs are not of this form.
This is not surprising as the asymptotic estimates and intuition in the proof of Theorem~\ref{thr:minIw_as} is only about big order behaviour, and also for trees the broom was not extremal for small $n$.
For $5 \le n \le 9$, the extremal graphs are presented in Figure~\ref{fig:extG_minIw}.

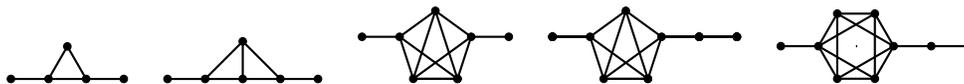
\begin{figure}[h]
\centering
\begin{tikzpicture}
\foreach \x in {0,0.5,1,1.5}{\draw[fill] (\x,0) circle (0.05);
}
\draw[fill] (0.75,0.433) circle (0.05);
\foreach \x in {0,0.5,1}{\draw[thick](\x+0.5,0) -- (\x,0);}
\foreach \x in {0.5,1}{\draw[thick](0.75,0.433) -- (\x,0);
}
\end{tikzpicture}
\quad
\begin{tikzpicture}
\foreach \x in {0.5}{\draw[fill] (90:\x) circle (0.05);
\draw[thick](90:\x-0.5) -- (90:\x);
\draw[thick](90:\x) -- (0:\x-1);
\draw[thick](90:\x) -- (0:\x);
}
\draw[fill] (0:-1) circle (0.05);
\foreach \x in {-0.5,0,...,1}{\draw[fill] (0:\x) circle (0.05);
\draw[thick](0:\x-0.5) -- (0:\x);
\draw[color=white](0,1)--(0,1);
}
\end{tikzpicture}
\quad
\begin{tikzpicture}
\foreach \x in {-54,18,...,306}{\draw[fill] (\x:0.5) circle (0.05);
\draw[thick](\x-72:0.5) -- (\x:0.5);
}
\foreach \x in {-54,18,90,234}{
\draw[thick](\x-144:0.5) -- (\x:0.5);
}
\draw[thick](-0.4755,0.1545) -- (-0.9755,0.1545);
\draw[thick](0.4755,0.1545) -- (0.9755,0.1545);
\draw[fill] (0.9755,0.1545) circle (0.05);
\draw[fill] (-0.9755,0.1545) circle (0.05);
\end{tikzpicture}
\quad
\begin{tikzpicture}
\foreach \x in {-54,18,...,306}{\draw[fill] (\x:0.5) circle (0.05);
\draw[thick](\x-72:0.5) -- (\x:0.5);
}
\foreach \x in {-54,18,90,234}{
\draw[thick](\x-144:0.5) -- (\x:0.5);

\draw[thick](-0.4755,0.1545) -- (-0.9755,0.1545);
\draw[thick](1.4755,0.1545) -- (0.4755,0.1545);
\draw[fill] (1.4755,0.1545) circle (0.05);
\draw[fill] (0.9755,0.1545) circle (0.05);
\draw[fill] (-0.9755,0.1545) circle (0.05);
}
\end{tikzpicture}
\quad
\begin{tikzpicture}
\foreach \x in {0,60,...,300}{
\draw[thick](\x:0.5) -- (\x+60:0.5);
\draw[fill] (\x:0.5) circle (0.05);
\draw[thick](\x:0.5) -- (\x+120:0.5);
}
\foreach \x in {0,60}{
\draw[thick](\x:0)--(-\x:0);
}
\foreach \x in {-1,0.5,1}{
\draw[thick](\x,0)--(\x+0.5,0);
}
\foreach \x in {-1,1.5,1}{
\draw[fill] (\x,0) circle (0.05);
}
\end{tikzpicture}
    \caption{Graphs with the minimum Wiener-entropy for $5 \le n \le 9$}
    \label{fig:extG_minIw}
\end{figure}

If the extremal graphs are the concatenation of a path and a clique, they would again be defined by a step-wise increasing function $k(n)=n^{1/2+o(1)}$ only taking integers.
One can expect that if it would be plausible, the optimal function $k(x)$ would be more continuous.
By connecting the end vertex of the path with only a portion of the vertices of the clique, there is this more continuous behaviour.
Recall that $G_{n,k,j}$ is the disjoint union of a path $P_k$ and a clique $K_{n-k}$, with $j$ vertices of the clique connected to the end vertex of the path (Figure~\ref{fig:graphGnkj}).

Restricted to the class of graphs of the form $G_{n,k,j}$, we computed the extremal graphs for small $n.$
For $16\le n \le 94$, we notice the behaviour that one can expect.
When $n$ is growing, step-wise $k$ grows and in these steps $j$ decreases. 
For larger $n$, we observe that $j=1$ appears more often, e.g., when $209 \le n \le 221$ this happens for $6$ out of the $13$ values for which $k(n)=26.$
All of these values are presented in Table~\ref{table:minIW_smalln}.

\begin{table}[h]
\centering
\begin{tabular}{|c|c|c|}
	\hline
	\textbf{$n$} & \textbf{$(k,j)$} & \textbf{$I_w(G_{n,k,j})$} \\
	\hline
 32& $(8, 22)$ & $ 4.8418782994
$\\ 33& $(8, 20)$ & $ 4.8824114556
$\\34& $(8, 18)$ & $ 4.9217394089
$\\35& $(8, 15)$ & $ 4.9599202002
$\\36& $(8, 12)$ & $ 4.9970026044
$\\37& $(8, 8)$ & $ 5.0330361551
$\\38& $(8, 4)$ & $ 5.0680644063
$\\39& $(9, 26)$ & $ 5.1020833397
$\\40& $(9, 23)$ & $ 5.1352102662
$\\41& $(9, 20)$ & $ 5.1675123079
$\\42& $(9, 16)$ & $ 5.1990223046
$\\43& $(9, 12)$ & $ 5.2297674906
$\\44& $(9, 8)$ & $ 5.2597769036
$\\45& $(9, 3)$ & $ 5.2890774027
$\\
46& $( 10, 31)$ & $ 5.3176708476$\\
	\hline
	\end{tabular}
	\quad
	\begin{tabular}{|c|c|c|}
	\hline
	\textbf{$n$} & \textbf{$(k,j)$} & \textbf{$I_w(G_{n,k,j})$} \\
	\hline
208&$(25,1)$& $7.2287884533$\\
209&$(26,159)$& $7.2347291497$\\
210&$(26,133)$& $7.2406346487$\\
211&$(26,108)$& $7.246505897$\\
212&$(26,84)$& $7.2523437764$\\
213&$(26,61)$& $7.2581490247$\\
214&$(26,38)$& $7.2639222854$\\
215&$(26,16)$& $7.2696641255$\\
216&$(26,1)$& $7.2753755053$\\
217&$(26,1)$& $7.281063551$\\
218&$(26,1)$& $7.2867307557$\\
219&$(26,1)$& $7.2923773056$\\
220&$(26,1)$& $7.2980033842$\\
221&$(26,1)$& $7.3036091723$\\
222&$(27,175)$& $7.3091923114$\\
	\hline
	\end{tabular}
		\caption{Minimum value of Wiener-entropy among graphs of the form $G_{n,k,j}$ for $32 \le n \le 46$ an $208 \le n \le 222$}\label{table:minIW_smalln}
\end{table}

Maybe surprisingly, for large $n$, it seems that $j=1$ is always true and so there is some stability result or discretization for large values of $n$, which was not there for the smaller values.
Noting that the size $m=\binom{n-k}{2}+j+k-1,$ one can also plot the value $I_w(G_{n,k,j})$ in terms of the size, and expecting some monotonicity below and above the optimal choice. This seems to be true in large regions, but is not always true. As an example, when $n=48$, then the optimal size is $m=736$ ($k=10$ and $j=24$) and around this value the function $I_w$ behaves nice, but for $k=38$ we do not have monotonicity in terms of $j$. This is presented in Figure~\ref{fig:n48m}.

\begin{figure}[htbp]
\centering
{
\includegraphics[width=7cm]{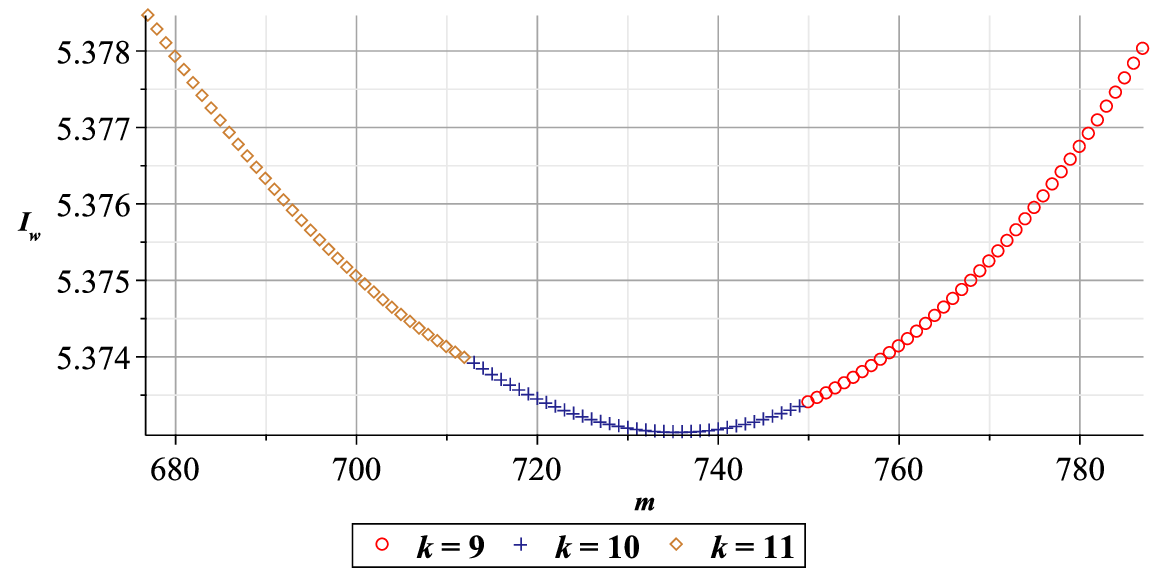}
}
\quad
{
\includegraphics[width=7cm]{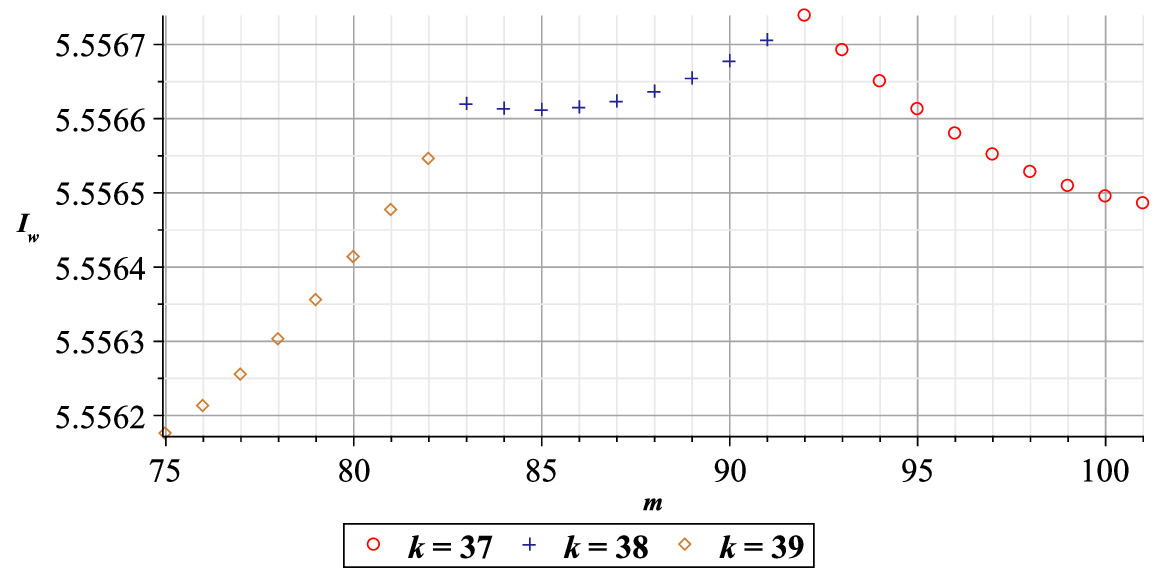}
}
\caption{Plots of $I_w(G_{48,k,j})$ for $k \sim 10$ and $k \sim 38$}
    \label{fig:n48m}
\end{figure}


For a given $n$, let $k'$ be the optimal choice for which there is a $j$ such that $G_{n,k',j}$ attains the minimum of the Wiener-entropy among all choices of graphs of the form $G_{n,k,j}$. Then for $k=k'\pm 1,$ we observed the same behaviour as was the case with $n=48$ for larger values.
Nevertheless, when plotting $I_w(G_{n,k',j})$ for larger $n$ as a function of $j$, there are examples with multiple local minima.

\begin{question}
    Can one explain (give intuition on) the difference in behaviour, depending on the order,
    for the graphs of the form $G_{n,k,j}$ minimizing $I_w(G)$.
\end{question}

\begin{figure}[htbp]
\centering
{
\begin{minipage}[t]{0.48\linewidth}
\centering
\includegraphics[width=7.3cm]{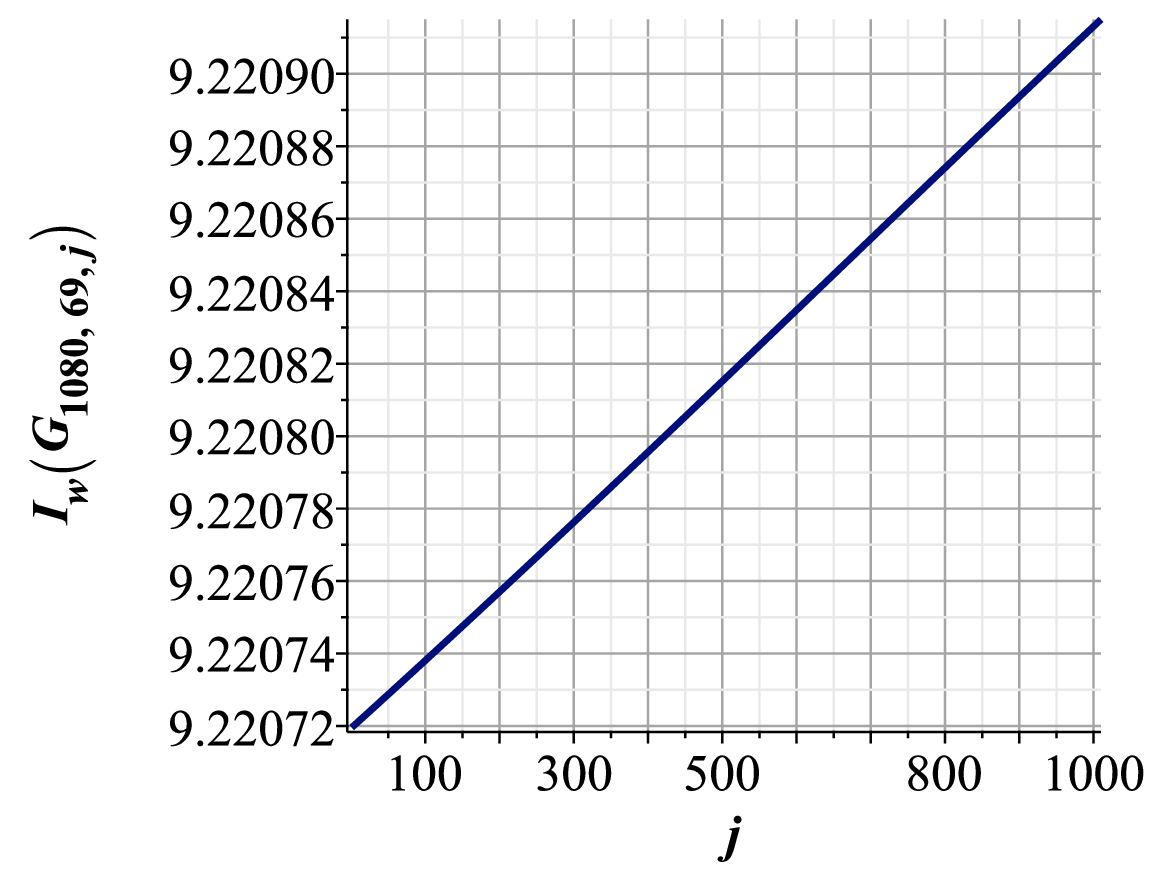}
\end{minipage}%
}%
{
\begin{minipage}[t]{0.48\linewidth}
\centering
\includegraphics[width=7.3cm]{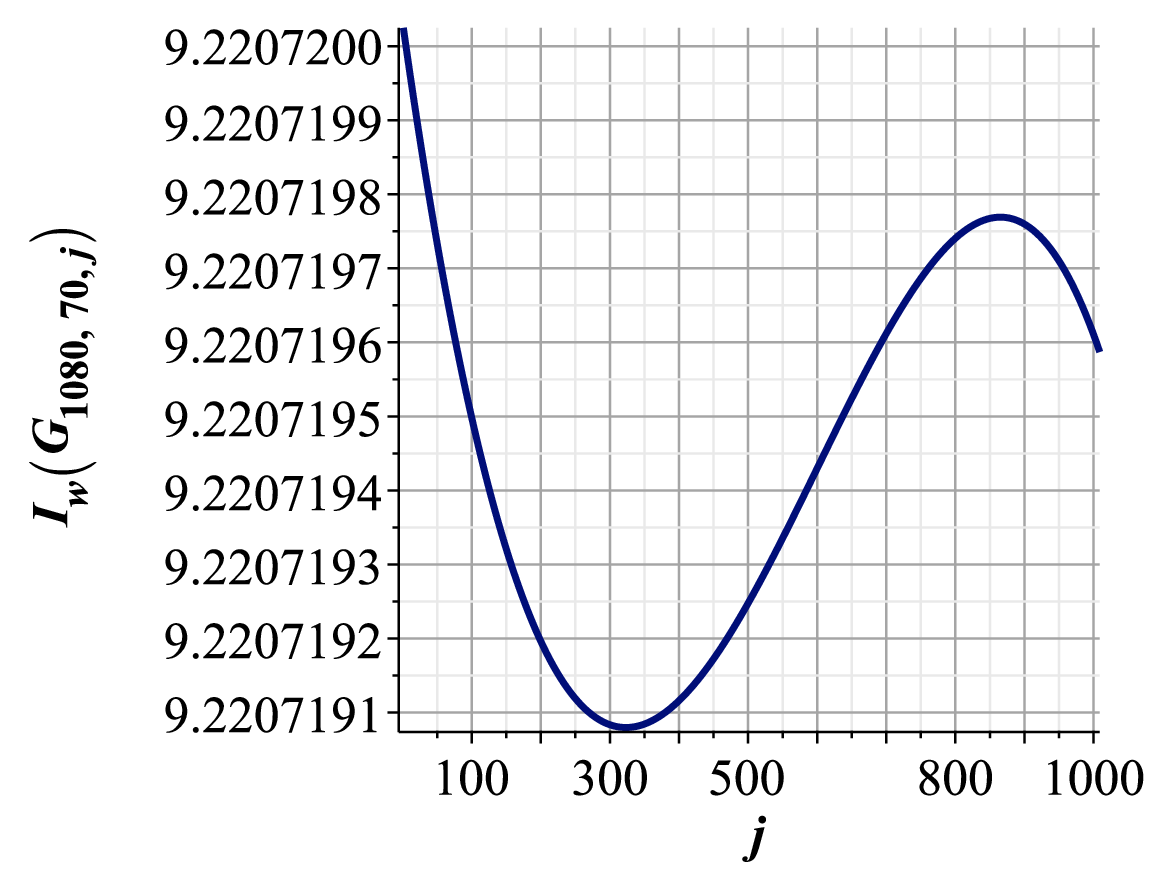}
\end{minipage}
}%
{
\begin{minipage}[t]{0.48\linewidth}
\centering
\includegraphics[width=7.3cm]{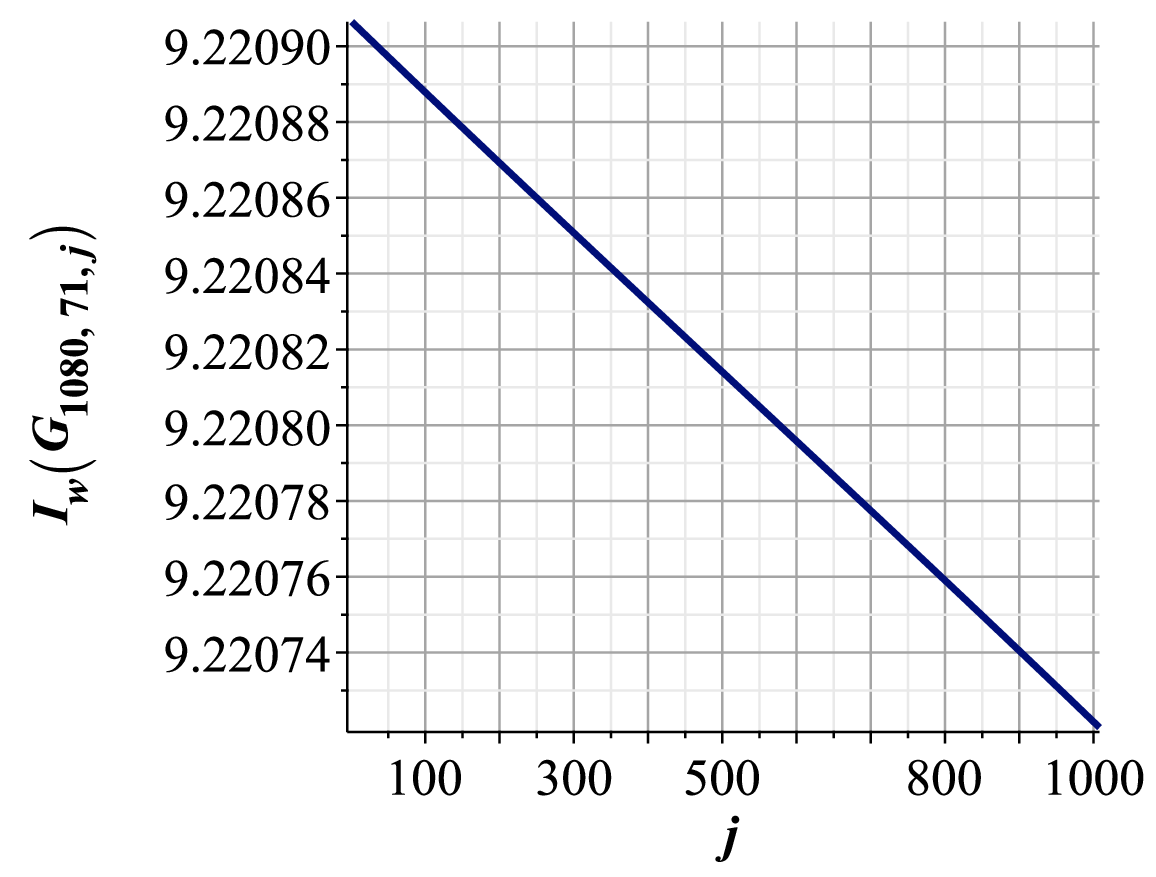}
\end{minipage}
}%
\centering
\caption{Plots of $I_w(G_{1080,k,j})$ for $k=69,70,71$}
\label{fig:n1080kj}
\end{figure}


We expect that the extremal graphs are of the form $G_{n,k,j}$ from a reasonable small constant onwards and conjecture that the extremal graphs for large $n$ are indeed the concatenation of a path and clique, i.e., a graph of the form $G_{n,k,1}.$

\begin{conj}
    There exists a value $n_0$ such that for all $n \ge n_0$, among all trees and graphs of order $n$, the Wiener-entropy is minimized by respectively a broom and a $G_{n,k,1}$.
\end{conj}

Based on a verification among the graphs of the form $G_{n,k,j}$ it seems plausible that $n_0=1270.$ The sporadic examples for which $1000\le n\le 2540$\footnote{Computations have been done within a restricted range, based on an assumption of monotonicity in $k$, see~\url{https://github.com/yndongmath/wiener-entropy/tree/main/CalE}. } and $j>1$ are given in Table~\ref{table:minIW_jnot1}.

For $n \ge 16$, the graphs of the form $G_{n,k,j}$ with the minimum value of $I_w(G_{n,k,j})$ have been computed and are summarized in Table~\ref{table:minIW}, for some powers of $2.$
For these powers of $2$, the value for the Wiener-entropy can be easily compared with $\log(n)$ and as such, we observe that the $o(1)$ part in Theorem~\ref{thr:minIw_as} tends to zero rather slowly.

\begin{table}[h]

\begin{minipage}[b]{.45\linewidth}
\begin{center}
    \begin{tabular}{|c|c|c|}
	\hline
	\textbf{$n$} & \textbf{$(k,j)$} & \textbf{$I_w(G_{n,k,j})$} \\
	\hline
1003&$(67,401)$& $9.1328643808$\\
1004&$(67,75)$& $9.1340468031$\\
1029&$(68,152)$& $9.163275375$\\
1054&$(69,389)$& $9.1917887671$\\
1055&$(69,29)$& $9.1929126061$\\
1080&$(70,323)$& $9.2207190796$\\
1133&$(72,112)$& $9.2775546382$\\
1269&$(77,37)$& 
$9.4118343668$\\
	\hline
	\end{tabular}\\
\subcaption{Values $1000\le n$ for which the $G_{n,k,j}$ attaining the minimum value of $I_w$ satisfies $j>1$ }\label{table:minIW_jnot1}
\end{center}
\end{minipage}\quad\begin{minipage}[b]{.45\linewidth}

\begin{center}
\begin{tabular}{|c|c|c|}
	\hline
	\textbf{$n$} & \textbf{$(k,j)$} & \textbf{$I_w(G_{n,k,j})$} \\
	\hline
	16 & $(5,9)$ & $3.9126433225$ 	 \\ 	
	32 & $(8,22)$ & $4.8418782994$ \\ 	
    64 & (12,26) & $5.744804111$\\
    128& (19,69) & 6.624593606\\
    256& (29,4)& $7.4845154156$\\
    512& (44,1)  & $8.32786753$\\
    1024& (67,1) & $9.1574755626$\\
    2048& (101,1) & $9.9757653248$\\
    4096& (152,1) & $10.7847443225$\\
    8192& (225,1) & 11.5860993918\\
	\hline
	\end{tabular}\\
	\subcaption{Minimum value of Wiener-entropy among graphs of the form $G_{n,k,j}$ for some powers of $2$}\label{table:minIW}
\end{center}
\end{minipage}
\caption{Minimum of $I_w(G_{n,k,j})$ for some special cases}
\end{table}

\section*{Acknowledgement}

We thank the referees for their careful reading and suggestions for improvements and pointing out reference~\cite{AL11}.

\paragraph{Open access statement.} For the purpose of open access,
a CC BY public copyright license is applied
to any Author Accepted Manuscript (AAM)
arising from this submission.

\bibliographystyle{abbrv}
\bibliography{Iw}

\section*{Appendix}

\begin{table}[h]
\setlength{\tabcolsep}{5mm} 
\def\arraystretch{1.5} 
\centering
  \begin{tabular}{|m{0.5cm}<{\centering}|m{5cm}<{\centering}|m{0.5cm}<{\centering}|m{5cm}<{\centering}|}
      \hline

      $n$    &   $T_{s}(n)$    &   $n$  &  $T_{s}(n)$       \\ \hline

   $3$   &
\begin{tikzpicture}
\draw[fill] (0:-0.5) circle (0.05);
\foreach \x in {0,0.5}{\draw[fill] (0:\x) circle (0.05);
\draw[thick](0:\x-0.5) -- (0:\x);
}
\end{tikzpicture}
& $4$  &
\begin{tikzpicture}
\foreach \x in {0.5}{\draw[fill] (90:\x) circle (0.05);
\draw[thick](90:\x-0.5) -- (90:\x);
}
\draw[fill] (0:-0.5) circle (0.05);
\foreach \x in {0,0.5}{\draw[fill] (0:\x) circle (0.05);
\draw[thick](0:\x-0.5) -- (0:\x);
\draw[color=white](0,1)--(0,1);
}
\end{tikzpicture}    \\ \hline

$5$  &
\begin{tikzpicture}
\draw[fill] (0:-1) circle (0.05);
\foreach \x in {-0.5,0,...,1}{\draw[fill] (0:\x) circle (0.05);
\draw[thick](0:\x-0.5) -- (0:\x);
}
\end{tikzpicture}
& $6$   &
\begin{tikzpicture}
\foreach \x in {0.5}{\draw[fill] (90:\x) circle (0.05);
\draw[thick](90:\x-0.5) -- (90:\x);
}
\draw[fill] (0:-1) circle (0.05);
\foreach \x in {-0.5,0,...,1}{\draw[fill] (0:\x) circle (0.05);
\draw[thick](0:\x-0.5) -- (0:\x);
\draw[color=white](0,1)--(0,1);
}
\end{tikzpicture}   \\ \hline

$7$ &
\begin{tikzpicture}
\foreach \x in {0.5,1}{\draw[fill] (45:\x) circle (0.05);
\draw[thick](45:\x-0.5) -- (45:\x);
}
\draw[fill] (0:-1) circle (0.05);
\foreach \x in {-0.5,0,...,1}{\draw[fill] (0:\x) circle (0.05);
\draw[thick](0:\x-0.5) -- (0:\x);
\draw[color=white](0,1)--(0,1);
}
\end{tikzpicture}
  & $8$ &
\begin{tikzpicture}
\foreach \x in {0.5,1}{\draw[fill] (45:\x) circle (0.05);
\draw[thick](45:\x-0.5) -- (45:\x);
}
\draw[fill] (0:-1.5) circle (0.05);
\foreach \x in {-1,-0.5,...,1}{\draw[fill] (0:\x) circle (0.05);
\draw[thick](0:\x-0.5) -- (0:\x);
\draw[color=white](0,1)--(0,1);
}
\end{tikzpicture}  
    \\ \hline
$9$  &
\begin{tikzpicture}
\foreach \x in {0.5,1}{\draw[fill] (45:\x) circle (0.05);
\draw[thick](45:\x-0.5) -- (45:\x);
}
\draw[fill] (0:-1.5) circle (0.05);
\foreach \x in {-1,-0.5,...,1.5}{\draw[fill] (0:\x) circle (0.05);
\draw[thick](0:\x-0.5) -- (0:\x);
\draw[color=white](0,1)--(0,1);
}
\end{tikzpicture}  
 & $10$ &   
 \begin{tikzpicture}
\foreach \x in {0.5,1,1.5}{\draw[fill] (30:\x) circle (0.05);
\draw[thick](30:\x-0.5) -- (30:\x);
}
\draw[fill] (0:-1.5) circle (0.05);
\foreach \x in {-1,-0.5,...,1.5}{\draw[fill] (0:\x) circle (0.05);
\draw[thick](0:\x-0.5) -- (0:\x);
\draw[color=white](0,1)--(0,1);
}
\end{tikzpicture}    \\ \hline
 $11$   & \begin{tikzpicture}
\foreach \x in {0,45,...,135}{\draw[fill] (\x:0.5) circle (0.05);
\draw[thick](0:0) -- (\x:0.5);}
\draw[fill] (0:-2) circle (0.05);
\foreach \x in {-1.5,-1,...,1.5}{\draw[fill] (0:\x) circle (0.05);
\draw[thick](0:\x-0.5) -- (0:\x);
}
\end{tikzpicture}  & $12$ &   
 \begin{tikzpicture}
\foreach \x in {0,45,...,135}{\draw[fill] (\x:0.5) circle (0.05);
\draw[thick](0:0) -- (\x:0.5);}
\draw[fill] (0:-2) circle (0.05);
\foreach \x in {-1.5,-1,...,2}{\draw[fill] (0:\x) circle (0.05);
\draw[thick](0:\x-0.5) -- (0:\x);
\draw[color=white](0,1)--(0,1);
}
\end{tikzpicture}  \\ \hline
 $13$  &  \begin{tikzpicture}
\foreach \x in {0,36,...,144}{\draw[fill] (\x:0.5) circle (0.05);
\draw[thick](0:0) -- (\x:0.5);}
\draw[fill] (0:-2) circle (0.05);
\foreach \x in {-1.5,-1,...,2}{\draw[fill] (0:\x) circle (0.05);
\draw[thick](0:\x-0.5) -- (0:\x);
}
\end{tikzpicture} & $14$ &   
 \begin{tikzpicture}
\foreach \x in {0,30,...,160}{\draw[fill] (\x:0.5) circle (0.05);
\draw[thick](0:0) -- (\x:0.5);}
\draw[fill] (0:-2) circle (0.05);
\foreach \x in {-1.5,-1,...,2}{\draw[fill] (0:\x) circle (0.05);
\draw[thick](0:\x-0.5) -- (0:\x);
\draw[color=white](0,1)--(0,1);
}
\end{tikzpicture}  \\ \hline
 $15$   &
 \begin{tikzpicture}
\foreach \x in {0,30,...,160}{\draw[fill] (\x:0.5) circle (0.05);
\draw[thick](0:0) -- (\x:0.5);}
\draw[fill] (0:-2) circle (0.05);
\foreach \x in {-1.5,-1,...,2.5}{\draw[fill] (0:\x) circle (0.05);
\draw[thick](0:\x-0.5) -- (0:\x);
\draw[color=white](0,1)--(0,1);
}
\end{tikzpicture}
& $16$ &  \begin{tikzpicture}
\foreach \x in {0,26,...,156}{\draw[fill] (\x:0.5) circle (0.05);
\draw[thick](0:0) -- (\x:0.5);}
\draw[fill] (0:-2) circle (0.05);
\foreach \x in {-1.5,-1,...,2.5}{\draw[fill] (0:\x) circle (0.05);
\draw[thick](0:\x-0.5) -- (0:\x);
\draw[color=white](0,1)--(0,1);
}
\end{tikzpicture}   \\ \hline
 $17$   &
\begin{tikzpicture}
\draw[color=white](0,1)--(0,1);
\foreach \x in {0,33,...,330}{\draw[fill] (\x:0.5) circle (0.05);
\draw[thick](0:0) -- (\x:0.5);}
\foreach \x in {0.5,1,...,3}{\draw[fill] (0:\x) circle (0.05);
\draw[thick](0:\x-0.5) -- (0:\x);
}
\end{tikzpicture} 
& $18$ &  \begin{tikzpicture}
\draw[color=white](0,1)--(0,1);
\foreach \x in {0,30,...,330}{\draw[fill] (\x:0.5) circle (0.05);
\draw[thick](0:0) -- (\x:0.5);}
\foreach \x in {0.5,1,...,3}{\draw[fill] (0:\x) circle (0.05);
\draw[thick](0:\x-0.5) -- (0:\x);
}
\end{tikzpicture}   \\ \hline
  \end{tabular}
\caption{Trees with the minimum Wiener-entropy}
\vskip 4mm
\label{Tab:extT_minIw}
\end{table}

\end{document}